\documentclass[11pt]{article}
\usepackage{amssymb}
\usepackage{amsmath}
\usepackage{amscd}
\usepackage{amsbsy}
\usepackage{amsfonts}
\usepackage{theorem}
\usepackage{hyperref}
\usepackage{bbm}

\input xy
\xyoption{all}

\newcommand{\zz}{{\Bbb Z}}

\newcommand{\pp}{{\Bbb P}}

\newcommand{\ff}{{\Bbb F}}

\newcommand{\ddim}{\operatorname{dim}}
\newcommand{\ddeg}{\operatorname{deg}}

\newcommand{\kker}{\operatorname{Ker}}

\newcommand{\Homi}{\underline{\operatorname{Hom}}}
\newcommand{\Hom}{\operatorname{Hom}}
\newcommand{\op}[1]{\operatorname{#1}}
\newcommand{\kbar}{\overline{k}}

\newcommand{\ffi}{\varphi}

\newcommand{\la}{\langle}
\newcommand{\ra}{\rangle}
\newcommand{\lva}{\langle\!\langle}   
\newcommand{\rva}{\rangle\!\rangle}   
\newcommand{\row}{\rightarrow}

\newcommand{\low}{\leftarrow}
\newcommand{\lrow}{\longrightarrow}
\renewcommand{\leq}{\leqslant}
\renewcommand{\geq}{\geqslant}

\newcommand{\et}{\acute{e}t}

\newcommand{\nichego}[1]{}

\newcommand{\ov}[1]{\overline{#1}}
\newcommand{\un}[1]{\underline{#1}}
\newcommand{\wt}[1]{\widetilde{#1}}

\newcommand{\dmk}{\op{DM}(k)}

\newcommand{\dmgmkD}{\op{DM}_{gm}(k;\zz/2)}
\newcommand{\dmgmkF}[1]{\op{DM}_{gm}(k;#1)}
\newcommand{\dmgmEF}[2]{\op{DM}_{gm}(#1;#2)}

\newcommand{\dmkF}[1]{\op{DM}(k;#1)}
\newcommand{\dmEF}[2]{\op{DM}(#1;#2)}
\newcommand{\hii}{{\cal X}}
\newcommand{\whii}{\widetilde{{\cal X}}}

\newcommand{\moco}[2]{H_{{\cal M}}^{#1,#2}}
\newcommand{\moho}[2]{H^{{\cal M}}_{#1,#2}}

\newcommand{\gm}{{\Bbb G}_m}
\newcommand{\calg}{{\cal G}}
\newcommand{\calh}{{\cal H}}
\newcommand{\comg}{\ov{{\cal G}}}

\newcommand{\Qed}{\hfill$\square$\smallskip}

\newenvironment{proof}{\noindent{\it Proof}:}{\vskip 5mm}

\newtheorem{prop}{Proposition}[section]{\bf}{\it}
\newtheorem{thm}[prop]{Theorem}{\bf}{\it}
\newtheorem{lem}[prop]{Lemma}{\bf}{\it}
{\bf}{\it}
\newtheorem{defi}[prop]{Definition}{\bf}{\it}
{\bf}{\it}
{\bf}{\it}
\newtheorem{exa}[prop]{Example}{\bf}{\it}
\newtheorem{rem}[prop]{Remark}{\bf}{}
\newtheorem{que}[prop]{Question}{\bf}{\it}

{\bf}{\it}
{\bf}{\it}
\newtheorem{cor}[prop]{Corollary}{\bf}{\it}
{\bf}{\it}

\begin{document}

\title{The Balmer spectrum of Voevodsky motives and pure symbols}
\author{Alexander Vishik\footnote{School of Mathematical Sciences, University
of Nottingham}}
\date{}
\maketitle

\begin{abstract}
In this article we introduce invariants of points of the Balmer spectrum of the Voevodsky motivic category 
$\dmkF{\ff_2}$ whose values are {\it light Rost cycle submodules} of the module of pure symbols in Milnor's K-theory (mod $2$). As an application, we show that {\it isotropic points} of the Balmer spectrum are closed. 
We also introduce the notion of points of a {\it boundary type} and show that this class contains isotropic points, but not the etale one.
\end{abstract}

\section{Introduction}

The most important features of the Voevodsky motivic category $\dmk$ are encoded in the Balmer spectrum $\op{Spc}$ \cite{Bal-1} of it. This generalisation of the Zariski spectrum of a commutative ring is a ringed topological space whose points are prime thick tensor ideals of the category.
While the topological counterpart of the motivic category is the derived category of abelian groups whose Balmer spectrum (of the compact part) is simply $\op{Spec}(\zz)$, the Voevodsky category itself is substantially more complicated, which is reflected in the structure of the Balmer spectrum. 
To start with, the Galois group of the base field $k$ enters the game. The category of etale motives (which is, in a sense, a simplified version of $\dmk$) is described in terms of it. This group classifies possible zero-dimensional varieties over $k$ whose motives generate the
subcategory $\op{DAM}(k)$ of Artin motives. Balmer
and Gallauer in \cite{BG} described completely the spectrum of this subcategory in terms of subgroups of $Gal(\ov{k}/k)$ and shown that it is quite non-trivial. This bounds from below the complexity
of the Balmer spectrum of $\dmk$, as it naturally surjects 
to that of $\op{DAM}(k)$. But the varieties of positive dimension can't be reduced to zero-dimensional ones and the fibers of the projection $\op{Spc}(\dmk^c)\row
\op{Spc}(\op{DAM}(k)^c)$ are expected to be ``large''.

The spectra $\op{Spec}(E)$ of finitely generated field
extensions $E/k$ may be considered as ``atomic'' objects 
of the algebro-geometric world (compare with a single atomic object = point in topology), which leads to the
idea of {\it isotropic realisations} \cite{Iso}. Such realisations
$$
\psi_{p,E}:\dmk\row\dmEF{\wt{E}/\wt{E}}{\ff_p}
$$
are parametrized by the choice of a prime $p$ and a 
$p$-equivalence class of field extensions of $k$ (where two extensions are $p$-equivalent, if $p$-isotropy of $k$-varieties over them is equivalent). The target here is the {\it isotropic motivic category} over a {\it flexible} field, which is way simpler than the global Voevodsky category. It is shown in \cite[Theorem 5.13]{INCHKm}
that the kernel ${\frak a}_{p,E}$ of the isotropic realisation is a prime ideal of $\dmk^c$. We get plenty of
{\it isotropic} points of the Balmer spectrum
(see \cite[Example 5.14]{INCHKm}).

In the current paper, we will impose certain coordinate system on the Balmer spectrum of Voevodsky category, which
is described in terms of pure symbols in Milnor's K-theory of finitely generated field extensions of $k$. The situation is understood better in the case of points of characteristic $2$. This is related to the fact that the $2$-isotropy of projective varieties is controlled by pure symbols $(mod\,2)$ over purely transcendental extensions of the base field - see \cite[Theorem 1.1]{VPS} (the same result is expected for odd primes). For this reason, we restrict our attention to the Voevodsky category
$\dmkF{\ff_2}$ with $\ff_2$-coefficients (whose Balmer spectrum constitutes the characteristic $2$ part of the spectrum of $\dmk$).

Our coordinate system is constructed with the help of certain ``test spaces'' - compact objects parametrized
by pure symbols $(mod\,2)$. Such compact objects, {\it Rost motives}, $M_{\alpha}$ were discovered by Rost
\cite{R-mPff} as direct summands of motives of Pfister quadrics $Q_{\alpha}$ (which are {\it norm-varieties} for symbols $\alpha$, in the sense, that the $2$-isotropy of $Q_{\alpha}$ is equivalent to the triviality of $\alpha$). 
Over the algebraic closure, $M_{\alpha}$ splits as a sum of two Tate-motives, but it is indecomposable as long as $\alpha$ is non-trivial. It is a motive of an affine quadric (the complement to a hyperplane section in $Q_{\alpha}$), that is, a motive of a non-split sphere. 
One may cut-out the two cells of the mentioned sphere and
get the {\it reduced Rost motive} $\wt{M}_{\alpha}$, which
in some sense, plays a dual role - see \cite[proof of Theorem 3.5]{Iso}. In particular, the thick tensor ideals in Voevodsky category generated by $M_{\alpha}$ and $\wt{M}_{\alpha}$ form orthogonals to each other - see Proposition \ref{orthog-MMt}. Here
$\alpha$, $Q_{\alpha}$, $M_{\alpha}$ and $\wt{M}_{\alpha}$
all carry exactly the same information. 

Since $M_{\alpha}\otimes\wt{M}_{\alpha}=0$, any prime ideal ${\frak a}$
of the Voevodsky category contains, at least, one of them, for any given $\alpha$. We may try to distinguish points
of the Balmer spectrum, by looking at which {\it Rost motives}, respectively, {\it reduced Rost motives} they contain. This identifies the etale point and isotropic points over flexible fields, but is not enough, in general.
To enhance our collection of test spaces, we consider pure symbols not only over the ground field, but also over all finitely generated extensions of it. It appears that the
{\it reduced Rost motive} is more suitable for the task.
We introduce new objects {\it extended reduced Rost motives} $\wt{M}_{\alpha,Y}$. These correspond to pure symbols $\alpha\in K^M_*(E)/2$, where $E/k$ is some finitely generated field extension, and are extensions of the
reduced Rost motive $\wt{M}_{\alpha}$ from $\op{Spec}(E)$
to some smooth $k$-neighbourhood $Y$ of it. Thus, it depends also on the neighbourhood $Y$, but the thick tensor ideal it generates depends on $\alpha$ only - see
Corollary \ref{ideal-indep-of-Y}. To each point ${\frak a}$ of the Balmer spectrum we assign two invariants 
${\cal G}({\frak a})$ and ${\cal H}({\frak a})$ with
values in the subsets of $\op{Pure}$ = the collection of
all pure symbols (mod $2$) over all finitely generated extensions of $k$. Namely, ${\cal G}({\frak a})$ contains
those symbols $(E,\alpha)$, for which $\wt{M}_{\alpha,Y}^{\perp}\subset{\frak a}$, while ${\cal H}({\frak a})$
contains those, for which $\wt{M}_{\alpha,Y}\in{\frak a}$ (note that, for symbols over the ground field, this is
the same as containing the {\it Rost motive}, respectively, the {\it reduced Rost motive}). 

We demonstrate that the introduced invariants are sufficiently
informative. In particular, they identify all isotropic points (as well as the etale one). We explicitly compute
these invariants for isotropic points - Proposition 
\ref{GH-iso-empty} and show that the isotropic ideal ${\frak a}_F$ is generated by what is prescribed by ${\cal G}({\frak a}_F)$ - Proposition
\ref{iso-ideal-descr}. This implies that {\it isotropic points} of the Balmer spectrum are closed - Theorem 
\ref{iso-closed}. In particular, there are no specialisation relations among them. 

We may consider $\ov{{\cal G}}({\frak a})=\op{Pure}\backslash{\cal G}({\frak a})$. Then $\ov{{\cal G}}({\frak a})\subset{\cal H}({\frak a})$.
We show that there is a rich structure on $\ov{{\cal G}}({\frak a})$ and ${\cal H}({\frak a})$. Namely, these are not just subsets of $\op{Pure}$, but {\it light Rost cycle submodules}, i.e. stable under restrictions of fields, multiplication by $\op{Pure}$ and residues with respect to DVRs (that is, all the operations of a {\it Rost cycle module} \cite{RCM}, aside from transfers) - see Theorem 
\ref{Gbar-H-LRCM}. 

To every point ${\frak a}$ of the Balmer spectrum we assign a $\stackrel{2}{\sim}$-equivalence class $K({\frak a})$ of field extensions (defined in terms of ${\cal H}({\frak a})$)
and so, an isotropic point ${\frak a}_{K({\frak a})}$.
In the case of an isotropic point, it recovers the original one. This gives certain restrictions on ${\cal H}({\frak a})$ - Proposition \ref{H-any-iso}. 
We introduce the notion of a point of a {\it boundary type}, that is, a point for which ${\cal G}({\frak a})\cap{\cal H}({\frak a})=\emptyset$ (note, that 
${\cal G}({\frak a})\cup{\cal H}({\frak a})=\op{Pure}$
always). We show that all isotropic points are of a {\it boundary type}, while the {\it etale} point is not - see
Proposition \ref{GH-iso-empty} and Example \ref{exa-bt-iso-et} (the latter fact leads to the loss of double grading on Tate-motives in the etale realisation). This raises the natural question: is the {\it boundary type} the same as {\it closed}?

The article is organised as follows. In Section \ref{two}
we introduce the {\it Rost motives}, {\it reduced Rost motives} and, finally, the main object: the 
{\it extended reduced Rost motives} and study the functoriality of the latter with respect to various operations on pure symbols. In Section \ref{three} we introduce the ${\cal G}$-${\cal H}$-invariants of points of the Balmer spectrum and prove our main results.

\section{Test spaces}
 \label{two}

We will try to distinguish points of the Balmer spactrum, that is, prime ideals of the Voevodsky motivic category, with the help of 
some ``test spaces''. These will be parametrised by pure symbols over various field extensions of the base field, which will be assumed of characteristic zero.

\subsection{Rost motives}

Let $\alpha=\{a_1,a_2,\ldots,a_n\}\in K^M_*(k)/2$ be a pure symbol (mod $2$), $q_{\alpha}=\lva a_1,\ldots,a_n\rva$ be the respective Pfister form and $Q_{\alpha}$ - the Pfister quadric.

By the result of Rost \cite{R-mPff}, the motive of a Pfister quadric is
divisible by a motive of a large projective space:
$M(Q_{\alpha})=M_{\alpha}\otimes M(\pp^{2^{n-1}-1})$.
The quotient $M_{\alpha}$ is called the {\it Rost motive}. Over algebraic closure it splits into the sum of
just two Tates: $M_{\alpha}|_{\kbar}=T\oplus T(2^{n-1}-1)[2^n-2]$. Two of the four maps from this decomposition are defined already over the base field, which gives a diagram in $\op{DM}_{gm}(k,\ff_2)$:
\begin{equation}
\label{romo}
\xymatrix{
& T \ar[d]^{[1]} \ar[rd]^{[1]} & \\
M_{\alpha} \ar[ru] & R \ar[l] \ar[r] \ar@{}[lu]|-(0.22){\star} \ar@{}[rd]|-(0.22){\star} & \wt{M}_{\alpha} \ar[ld]^{[1]}\\
& T(d)[2d] \ar[u] \ar[lu] &
},
\end{equation}
where $d=2^{n-1}-1$, and $\wt{M}_{\alpha}$ is the {\it reduced Rost motive}. Here $\alpha$, $Q_{\alpha}$, $M_{\alpha}$, $\wt{M}_{\alpha}$ all carry the same information. In particular, $\wt{M}_{\alpha}$ vanishes
simultaneously with $\alpha$. This object has only $2^{n-1}$ non-trivial homology with respect to the homotopy $t$-structure, all isomorphic to the {\it Rost cycle
module} $\alpha\cdot k^M_*$, up to shift, where $k^M_*=K^M_*/2$ (recall, that the heart of the homotopy $t$-structure is the category of Rost cycle modules - see \cite{Deg} and \cite{RCM}).

\noindent {\bf The case $n=0$:} The only symbol $\alpha=\{\}=1$ of degree zero gives $\wt{M}_{\alpha}=k^M_*=\op{Cone}(\tau)$, where $\tau:T(-1)\row T$ corresponds to the only non-zero element of $\moco{0}{1}(k;\ff_2)$. In this case, $M_{\alpha}=0$.

Let $\hii_{\alpha}$ be the motive of the \v{C}ech 
simplicial scheme of the respective norm-variety - Pfister quadric $Q_{\alpha}$ (for $n=0$, $Q_{\alpha}=\emptyset$).
It is a $\otimes$-projector in $\op{DM}(k;\ff_2)$.
Let $\whii_{\alpha}$ be the reduced motive of our
\v{C}ech simplicial scheme - the complementary projector.
Then $M_{\alpha}\otimes\hii_{\alpha}=M_{\alpha}$ and
$\wt{M}_{\alpha}\otimes\whii_{\alpha}=\wt{M}_{\alpha}$ -
see the proof of \cite[Theorem 3.5]{Iso}. In particular,
$M_{\alpha}\otimes\wt{M}_{\alpha}=0$ (since the projectors are orthogonal to each other).

\begin{prop}
\label{orthog-MMt}
For $n>0$, the thick tensor ideals $\la M_{\alpha}\ra$ and $\la \wt{M}_{\alpha}\ra$ of $\op{DM}_{gm}(k;\ff_2)$ are $\otimes$-orthogonals to each other.
\end{prop}

\begin{proof}
Since $\hii_{\alpha}$ and $\whii_{\alpha}$ are complementary projectors, the localising tensor ideals $\la\hii_{\alpha}\ra_l$ and $\la\whii_{\alpha}\ra_l$ of $\op{DM}(k;\ff_2)$ are each other $\otimes$-orthogonals. Since $\hii_{\alpha}$ belongs to the localising subcategory generated by $M_{\alpha}$ (which, in turn is an extension of two shifted copies of the former), the compact part of $\la\hii_{\alpha}\ra_l$ is $\la M_{\alpha}\ra$.
The ideal $\la\wt{M}_{\alpha}\ra$ is the compact part of $\la\whii_{\alpha}\ra_l$. Indeed, $\wt{M}_{\alpha}$ is an extension of two (shifted) copies of $\whii_{\alpha}$, and if a compact $N$ is such that $N\otimes\whii_{\alpha}=N$, then $N\otimes\wt{M}_{\alpha}=\op{Cone}[-1](N(d)[2d+1]\row N)$. Note that $N$ vanishes in the etale topology (since $\whii_{\alpha}$ does), hence, by Proposition \ref{et-comp-fin-diag}, this extension is nilpotent and $N\in\la\wt{M}_{\alpha}\ra$. This also shows that $\la M_{\alpha}\ra^{\perp}=\la\wt{M}_{\alpha}\ra$ (since $M_{\alpha}$ is just another generator of $\la\hii_{\alpha}\ra_l$). 
Finally, if a compact $N$ is such that $N\otimes\wt{M}_{\alpha}=0$, then, over any field extension, the motivic homology of $N\otimes\whii_{\alpha}$ are $(d)[2d+1]$ periodic (for $d=2^{n-1}-1$), and so, is zero, as it vanishes below certain diagonal. Hence, $N\in\la M_{\alpha}\ra$ (the compact part of $\la\whii_{\alpha}\ra_l^{\perp}$) and thus, $\la\wt{M}_{\alpha}\ra^{\perp}=\la M_{\alpha}\ra$.
 \Qed
\end{proof}

We can try to distinguish points of $\op{Spc}(\op{DM}_{gm}(k;\ff_2))$ using {\it Rost} and {\it reduced Rost} motives
$M_{\alpha}$ and $\wt{M}_{\alpha}$ as test spaces.

\begin{exa}
 \begin{itemize}
  \item[$(1)$] Let $k$ be {\it flexible} and ${\frak a}_{k}$ be the {\it isotropic} point corresponding to the trivial extension $k/k$. By the proof of \cite[Corollary 3.3]{Iso}, this ideal is generated by the motives of anisotropic Pfister quadrics over $k$, i.e. ${\frak a}_{k}=\la M_{\alpha}\,|\,\forall\alpha\neq 0\ra$. Note, that this ideal doesn't contain any reduced Rost motive $\wt{M}_{\alpha}$ (for $\alpha\neq 0$), since the isotropic motivic category $\op{DM}(k/k;\ff_2)$ possesses double grading on Tate-motives, which would have been lost, if the kernel of
  isotropic realisation would contain both $M_{\alpha}$
  and $\wt{M}_{\alpha}$, for the same $\alpha$ - see diagram (\ref{romo}).
  \item[$(2)$] Let ${\frak a}_{et}$ be the kernel of 
  the etale realisation with $\ff_2$-coefficients. Then
  ${\frak a}_{et}=\la\op{Cone}(\tau)\ra$. Indeed, clearly, $\op{Cone}(\tau)$ vanishes in the etale realisation, since $\tau$ is inverted there. Conversely, if a compact object $U$ vanishes in the etale realisation, then so does $U\otimes U^{\vee}$, which then must have only finitely many diagonals in motivic homology (by the Beilinson-Lichtenbaum conjecture). This shows that $\tau\otimes id_U$ is nilpotent, and so, $U$ is a direct summand of an object which is an
  extension of finitely many (shifted) copies of $U\otimes\op{Cone}(\tau)$
(cf. Corollary \ref{et-comp-fin-diag}). 
  
  Thus, ${\frak a}_{et}=\la\op{Cone}(\tau)\ra=
  \la\wt{M}_{\alpha}\,|\,\forall\alpha\ra$. The only Rost
  motive $M_{\alpha}$ it contains is $M_{\{\}}=0$. This leads to the loss of double grading on Tate-motives in
  the etale realisation.
 \end{itemize}
\end{exa}

As we saw, in the above examples, it was sufficient to know which Rost (respectively, reduced Rost) motives were
contained in the ideal, to identify it. Unfortunately, in general, it is not enough. We have to consider pure symbols not only over the ground field, but also over finitely generated extensions of it, and generalise the notion of Rost motives.

\subsection{Extended reduced Rost motives}

Consider the set of all pure symbols over all finitely-generated extensions of $k$:
$$
\op{Pure}=\{(E,\alpha)\,|\,E/k-\text{fin. gen.},\,\alpha\in k^M_*(E)-\text{pure}\}\hspace{1mm}\subset\hspace{1mm} k^M_*, 
$$
considered as a subset of the Rost cycle module $k^M_*$.
It is closed under: 1) restriction of fields, 2) derivatives $\partial$ w.r.to DVRs and 3) action of ${\cal O}^*$, that is all the operations of a Rost cycle module (\cite{RCM}), aside from transfers. We will call 
such a structure a {\it light Rost cycle module} (in 
\cite{kerM}, the term {\it weak Rost cycle module} was used).

Let $(E,\alpha)\in\op{Pure}$. 
The quadric $Q_{\alpha}$ and the projector defining the 
Rost motive $M_{\alpha}$ over $E$ are defined in some
smooth $k$-neighbourhood $Y$ of $\op{Spec}(E)$ producing a ``(relative) extended'' Rost motive $M_{\alpha/Y}\in\op{DM}_{gm}(Y;\ff_2)$. In particular, $\alpha$ is unramified on $Y$.
Note that the composition $T_Y(d)[2d]\row M_{\alpha/Y}\row T_Y$ in the category of motives over $Y$ is zero, since it resides in the zero group $\moco{-2d}{-d}(Y;\ff_2)$, for $n>1$,
respectively, is equal to $2=0$, for $n=1$. 
Moreover, the group $\moco{-2d-1}{-d}(Y;\ff_2)$ is zero as well, so the lifting in the (analogue of the)
diagram (\ref{romo}) is unique, and we get a canonical compact object
$\wt{M}_{\alpha/Y}\in\op{DM}_{gm}(Y;\ff_2)$. 

Let $Q_{\alpha/Y}\row Y$ be the smooth quadric fibration
extending $Q_{\alpha}$ and $\hii_{\alpha/Y}$ be the motive of the respective \v{C}ech simplicial scheme
considered as an object of $\op{DM}(Y;\ff_2)$. It is a $\otimes$-projector in this category. The complementary projector is given by the motive of the respective reduced simplicial scheme $\whii_{\alpha/Y}=\op{Cone}(\hii_{\alpha/Y}\row T_Y)\in\op{DM}(Y;\ff_2)$.
We have natural maps $\hii_{\alpha/Y}(d)[2d]\row M_{\alpha/Y}\row\hii_{\alpha/Y}$. The standard argument
of Voevodsky (see the proof of \cite[Theorem 4.4]{VoMil})
shows that these extend to a distinguished triangle:
$M_{\alpha/Y}=\op{Cone}[-1](\hii_{\alpha/Y}\row\hii_{\alpha/Y}(d)[2d+1])$ in $\op{DM}(Y;\ff_2)$. Tensoring the (analogue of) octahedron (\ref{romo}) by $\whii_{\alpha/Y}$ we get: 
$\wt{M}_{\alpha/Y}=\op{Cone}[-1](\whii_{\alpha/Y}(d)[2d+1]\row\whii_{\alpha/Y})$ - see the proof of \cite[Theorem 3.5]{Iso}. Denote as $\hii_{\alpha,Y}$,
$\whii_{\alpha,Y}$, $M_{\alpha,Y}$ and $\wt{M}_{\alpha,Y}$ the images of $\hii_{\alpha/Y}$, $\whii_{\alpha/Y}$,
$M_{\alpha/Y}$ and $\wt{M}_{\alpha/Y}$ under the natural functor $\pi_{\#}:\op{DM}(Y;\ff_2)\row\op{DM}(k;\ff_2)$. Here
$\wt{M}_{\alpha,Y}$ is our {\it extended reduced Rost motive}.

This object vanishes simultaneously with $\alpha$.

\begin{prop}
 \label{triv-erRm}
 For any field extension $L/k$,
 $$
 \wt{M}_{\alpha,Y}|_{L}=0\hspace{2mm}\Leftrightarrow
 \hspace{2mm}\alpha_{L(Y)}=0.
 $$
\end{prop}

\begin{proof}
 $(\low)$ If $\alpha_{L(Y)}=0$, then the projection
 $(Q_{\alpha/Y}\row Y)_L$ has a section in the generic point of $Y_L$, hence, in every point of $Y_L$, and so, the
 reduced motive $\whii_{\alpha,Y}|_L$ of the respective \v{C}ech simplicial scheme is zero $\Rightarrow$\ $\wt{M}_{\alpha,Y}=0$.
 
 \noindent
 $(\row)$ Let $\whii_{\alpha_{L(Y)}}\in\op{DM}(L(Y);\ff_2)$ be the ``generic fiber'' of $\whii_{\alpha/Y}|_L\in\op{DM}(Y_L;\ff_2)$. If $\alpha_{L(Y)}\neq 0$, then
 we have a non-zero element 
 $\tau^{-1}\alpha\in\moco{n+1}{n-1}(\whii_{\alpha_{L(Y)}};\ff_2)$ (see, for example, \cite{OVV}). The Brown-Gersten-Quillen differentials applied
 to this element land in motivic cohomology groups
 $\moco{n+2-2r}{n-1-r}(\whii_{\alpha_{L(y)}};\ff_2)$
 of special fibers. But these groups are zero, since the
 respective fibers are reduced \v{C}ech simplicial schemes of $n$-fold Pfister quadrics (see loc. cit.). Hence, the above element
 lifts to a non-zero element in
 $\moco{n+1}{n-1}(\whii_{\alpha,Y}|_L;\ff_2)$. Since
 $$
 \wt{M}_{\alpha,Y}=\op{Cone}[-1](\whii_{\alpha,Y}(d)[2d+1]\row\whii_{\alpha,Y})
 $$ 
 and motivic cohomology of
 $\whii_{\alpha,Y}$ can't be $(d)[2d+1]$-periodic (as it vanishes below the $2$-nd diagonal), we get that motivic
 cohomology of $\wt{M}_{\alpha,Y}|_L$ is non-zero.
 Thus, $\wt{M}_{\alpha,Y}|_L\neq 0$. 
 \Qed
\end{proof}

The object $\wt{M}_{\alpha,Y}$ depends on the choice of $Y$, but below we will show that the thick tensor ideal it generates depends on $\alpha$ only - see Corollary \ref{ideal-indep-of-Y}. For this we will need some tools.

Let $Y$ be a smooth variety over $k$ and 
$\dmgmEF{Y}{\Lambda}$ be the category of geometric motives over $Y$ with $\Lambda$-coefficients \cite{CD}. It is a tensor triangulated category with the natural adjoint pair of functors:
\begin{equation}
\label{adj-funct-Yk}
(\pi_Y)_{\#}:\dmgmEF{Y}{\Lambda}\rightleftarrows\dmgmkF{\Lambda}:\pi_Y^*,
\end{equation}
where $\pi_Y^*$ is tensor.
For any morphism $f:Z\row Y$ of smooth varieties we have
a tensor triangulated functor
$$
f^*:\dmgmEF{Y}{\Lambda}\row\dmgmEF{Z}{\Lambda}.
$$

\begin{defi}
 We will call an object $\bar{A}\in\dmgmEF{Y}{\Lambda}$ a 
 {\it co-algebra}, if it may be equipped with a co-associative co-multiplication $\bar{A}\stackrel{\Delta}{\lrow}\bar{A}\otimes\bar{A}$ with a co-unit $\nu:\bar{A}\row{\mathbbm{1}}$.
\end{defi}

Note that if $f:Z\row Y$ is a morphism of smooth varieties and $\bar{A}\in\dmgmEF{Y}{\Lambda}$ is a {\it co-algebra}, then $\bar{B}=f^*(\ov{A})$ is a {\it co-algebra} in
$\dmgmEF{Z}{\Lambda}$.

\begin{exa}
 Let $M\in\dmgmEF{Y}{\Lambda}$ be any object. Then
 $\bar{A}=M\otimes M^{\vee}$ is a co-algebra. The co-unit is the canonical map $\nu:M\otimes M^{\vee}\row{\mathbbm{1}}$, while $\Delta$ is given by
 $$
 id_M\otimes\nu^{\vee}\otimes id_{M^{\vee}}:M\otimes M^{\vee}=M\otimes {\mathbbm{1}}\otimes M^{\vee}\row
 M\otimes M^{\vee}\otimes M\otimes M^{\vee}.
 $$
\end{exa}

\begin{prop}
 \label{coalg-pullback}
 Let $f:Z\row Y$ be a morphism of smooth varieties, $\bar{A}\in\dmgmEF{Y}{\Lambda}$ be a {\it co-algebra} and $\bar{B}=f^*(\bar{A})\in\dmgmEF{Z}{\Lambda}$. Let
 $A=(\pi_Y)_{\#}(\bar{A})$ and $B=(\pi_Z)_{\#}(\bar{B})$.
 Then $B\in\la A\ra$.
\end{prop}

\begin{proof}
 We have the following commutative diagram
 $$
 \xymatrix{
 B \ar@{=}[r] & (\pi_Z)_{\#}(\bar{B}) \ar[r]^(0.43){\Delta_Z} \ar@{=}[rd] & (\pi_Z)_{\#}(\bar{B}\otimes\bar{B}) \ar[r] \ar[d]^{id\otimes\nu_Z} & B\otimes B \ar[r] \ar[d] & B\otimes A \ar[d] \\
 & & (\pi_Z)_{\#}(\bar{B}\otimes {\mathbbm{1}}) \ar[r] \ar@{=}[rd] & B\otimes M(Z) \ar[r] \ar[d] & B\otimes M(Y) \ar[d] \\
 & & & B \ar@{=}[r] & B
 }
 $$
where we identify $B\otimes B$ with $(\pi_{Z\times Z})_{\#}(\bar{B}\boxtimes\bar{B})$, $B\otimes M(Z)$ with $(\pi_{Z\times Z})_{\#}(\bar{B}\boxtimes {\mathbbm{1}})$, and similar for $B\otimes A$ and $B\otimes M(Y)$, and use the fact that $\bar{B}\otimes\bar{B}$ (respectively, $\bar{B}\otimes {\mathbbm{1}}$) is the diagonal restriction of $\bar{B}\boxtimes\bar{B}$ (respectively, $\bar{B}\boxtimes {\mathbbm{1}}$).
 Hence, $B$ is a direct summand of $B\otimes A$ and so, belongs to the thick tensor ideal generated by $A$.
 \Qed
\end{proof}

\begin{cor}
 \label{pi-dash-tensor-ideal}
Let $\bar{A}\in\dmgmEF{Y}{\Lambda}$ be a {\it co-algebra} and $A=(\pi_Y)_{\#}(\bar{A})$.
Then 
$$
(\pi_Y)_{\#}(\la\bar{A}\ra)=\la A\ra.
$$
\end{cor}

\begin{proof}
The inclusion $\supset$ follows from the projection formula. For the opposite inclusion, observe, that $\dmgmEF{Y}{\Lambda}$ is generated by the motives $M(z)$ of smooth maps
$z:Z\row Y$ over $Y$. In particular, the thick tensor ideal $\la\bar{A}\ra$ is generated as 
a thick subcategory by objects $\bar{A}\otimes M(z)$. 
Let $\bar{A}\boxtimes M(z)\in\dmgmEF{Y\times Y}{\Lambda}$ - be the external tensor product, and $\Delta: Y\row Y\times Y$ - the diagonal map. Then $\Delta^*(\bar{A}\boxtimes M(z))=\bar{A}\otimes M(z)$ and these objects are co-algebras. 
So, by Proposition \ref{coalg-pullback}, 
$(\pi_Y)_{\#}(\bar{A}\otimes M(z))\in\la(\pi_{Y\times Y})_{\#}(\bar{A}\boxtimes M(z))\ra=\la A\otimes M(Z)\ra\subset\la A\ra$. Since $(\pi_Y)_{\#}$ respects shifts, cones and direct summands, it maps $\la\bar{A}\ra$ to $\la A\ra$.
 \Qed
\end{proof}

Let $E/k$ be a finitely-generated extension, $\alpha\in K^M_n(E)/2$ be a pure symbol (mod 2), $Y$ be a smooth neighbourhood of $\op{Spec}(E)$, where $\alpha$ is unramified and $\wt{M}_{\alpha/Y}\in\dmgmEF{Y}{\ff_2}$ be the respective {\it (relative) extended reduced Rost motive}.  
It fits the following diagram:
\begin{equation}
\label{eromo}
\xymatrix{
& T \ar[d]^{[1]} \ar[rd]^{[1]} & \\
M_{\alpha/Y} \ar[ru] & R \ar[l] \ar[r] \ar@{}[lu]|-(0.22){\star} \ar@{}[rd]|-(0.22){\star} & \wt{M}_{\alpha/Y} \ar[ld]^{[1]}\\
& T(d)[2d] \ar[u] \ar[lu] &
},
\end{equation}
where $M_{\alpha/Y}$ is the (relative) extended Rost motive, which is a direct summand in the motive $M(Q_{\alpha}/Y)$ of the 
Pfister quadric fibration.
In $\dmEF{Y}{\ff_2}$, this motive is an extension of two shifted copies of the motive of the respective reduced
\v{C}ech simplicial scheme $\whii_{\alpha/Y}$:
\begin{equation}
\label{reromo-rCech}
\whii_{\alpha/Y}(d)[2d+1]\row\whii_{\alpha/Y}\row
\wt{M}_{\alpha/Y}[1]\row\whii_{\alpha/Y}(d)[2d+2],
\end{equation}
where $d=2^{n-1}-1$. The above diagram (\ref{eromo}) is self-dual with respect to $\Homi( -, T(d)[2d])$.

\begin{prop}
 \label{reromo-decomp}
 In $\dmgmEF{Y}{\ff_2}$ we have:
 $\wt{M}_{\alpha/Y}\otimes\wt{M}_{\alpha/Y}=\wt{M}_{\alpha/Y}[-1]\oplus\wt{M}_{\alpha/Y}(d)[2d+1]$.
\end{prop}

\begin{proof} Since $\whii_{\alpha/Y}$ is a tensor projector in $\dmgmEF{Y}{\ff_2}$, we get an exact triangle:
$$
\wt{M}_{\alpha/Y}(d)[2d]\stackrel{u}{\row}\wt{M}_{\alpha/Y}[-1]\row
\wt{M}_{\alpha/Y}\otimes\wt{M}_{\alpha/Y}\row\wt{M}_{\alpha/Y}(d)[2d+1].
$$
Here $u\in\Hom(\wt{M}_{\alpha/Y}(d)[2d+1],\wt{M}_{\alpha/Y})$ which sits in an exact sequence between the groups $\Hom(\whii_{\alpha/Y}(2d)[4d+2],\wt{M}_{\alpha/Y})$ and $\Hom(\whii_{\alpha/Y}(d)[2d],\wt{M}_{\alpha/Y})$ (all 
$\Hom$s are in $\dmEF{Y}{\ff_2}$). 
Since $(\wt{M}_{\alpha/Y})^{\vee}=\wt{M}_{\alpha/Y}(-d)[-2d]$, 
and there are no $\Hom$s from $\hii_{\alpha/Y}$ (shifted) to $\wt{M}_{\alpha/Y}$,
we may identify:
$\Hom(\whii_{\alpha/Y}(2d)[4d+2],\wt{M}_{\alpha/Y})$ with $\Hom(T(2d)[4d+2],\wt{M}_{\alpha/Y})$ and, finally, with
$\Hom(\wt{M}_{\alpha/Y},T(-d)[-2d-2])$.
From the diagram (\ref{eromo}) and the adjoint pair
(\ref{adj-funct-Yk}), we see that this group is zero,
since smooth varieties have no motivic cohomology (with $\ff_2$-coefficients) in negative round, or square degrees. Similarly, we see that:
$\Hom(\whii_{\alpha/Y}(d)[2d],\wt{M}_{\alpha/Y})=
\Hom(T(d)[2d],\wt{M}_{\alpha/Y})=
\Hom(\wt{M}_{\alpha/Y},T)$, and since
$\moco{-2d-1}{-d}(Y;\ff_2)=0$ and $\moco{1}{0}(Y;\ff_2)=0$,
the latter group may be identified with the group
$\op{coker}(\moco{0}{0}(Y;\ff_2)\row\moco{0}{0}(Q_{\alpha};\ff_2))$ which is zero. Thus, $u=0$ and so,
$\wt{M}_{\alpha/Y}\otimes\wt{M}_{\alpha/Y}$ splits into the direct sum as claimed.
 \Qed
\end{proof}

From this we immediately get:

\begin{cor}
 The thick subcategory of $\dmgmEF{Y}{\ff_2}$ generated by the Tate-twists of $\wt{M}_{\alpha/Y}$ coincides with the thick subcategory generated by the Tate twists of $\wt{M}_{\alpha/Y}\otimes (\wt{M}_{\alpha/Y})^{\vee}$.
 The thick tensor ideals of $\dmgmEF{Y}{\ff_2}$ generated by $\wt{M}_{\alpha/Y}$ and $\wt{M}_{\alpha/Y}\otimes (\wt{M}_{\alpha/Y})^{\vee}$ coincide.
\end{cor}

The advantage of $\wt{M}_{\alpha/Y}\otimes (\wt{M}_{\alpha/Y})^{\vee}$ in comparison to $\wt{M}_{\alpha/Y}$
is that it is a {\it co-algebra}.
This allows to show that the corresponding ideals are respected by the pull-backs.

Let $f:Z\row Y$ be a morphism of smooth varieties. Then
$f^*(\wt{M}_{\alpha/Y})=\wt{M}_{\beta/Z}$, where
$\beta=f^*(\alpha)$.
Let
$M_{\alpha,Y}=(\pi_Y)_{\#}(M_{\alpha/Y})$ and
$M_{\beta,Z}=(\pi_Z)_{\#}(M_{\beta/Z})$
be the respective
{\it extended reduced Rost motives} in $\dmgmkF{\ff_2}$.  

\begin{prop}
 \label{pullback-alpha}
 In the above situation, $\wt{M}_{\beta,Z}\in\la\wt{M}_{\alpha,Y}\ra$.
\end{prop}

\begin{proof}
 By Proposition \ref{reromo-decomp}, 
 $\la\wt{M}_{\alpha,Y}\ra=\la (\pi_Y)_{\#}(\wt{M}_{\alpha/Y}\otimes (\wt{M}_{\alpha/Y})^{\vee})\ra$. But 
 $\wt{M}_{\alpha/Y}\otimes (\wt{M}_{\alpha/Y})^{\vee}$ is a co-algebra. Hence, by Proposition \ref{coalg-pullback},
 $(\pi_Z)_{\#}(\wt{M}_{\beta/Z}\otimes (\wt{M}_{\beta/Z})^{\vee})=(\pi_Z)_{\#}f^*(\wt{M}_{\alpha/Y}\otimes (\wt{M}_{\alpha/Y})^{\vee})$ belongs to this ideal.
 By Proposition \ref{reromo-decomp}, so does
 $\wt{M}_{\beta,Z}=(\pi_Z)_{\#}(\wt{M}_{\beta/Z})$.
 \Qed
\end{proof}

Similarly, the co-algebra technique permits to see that our ideals are respected by the multiplication action by pure symbols.

\begin{prop}
 \label{divisible-Mt}
 Let $\alpha,\beta\in K^M_*(E)/2$ be pure symbols, such that $\alpha$ divides $\beta$ and both are unramified in the smooth neighbourhood $Y$ of $\op{Spec}(E)$. Then
 $\wt{M}_{\beta,Y}\in\la\wt{M}_{\alpha,Y}\ra$.
\end{prop}

\begin{proof}
 Let $\ddeg(\alpha)=m$, $d_{\alpha}\hspace{-0.8mm}=\hspace{-0.8mm}2^{m-1}-1$. 
 Let $C_{\alpha}=\wt{M}_{\alpha/Y}\otimes (\wt{M}_{\alpha/Y})^{\vee}$ and $C_{\beta}=\wt{M}_{\beta/Y}\otimes (\wt{M}_{\beta/Y})^{\vee}$ be the ``endomorphism co-algebras'' of
our objects in $\dmgmEF{Y}{\ff_2}$.
 Since $\alpha|\beta$,
 there is a morphism $Q_{\alpha}\row Q_{\beta}$ over $Y$,
 which shows that in $\dmgmEF{Y}{\ff_2}$,
 $\whii_{\beta/Y}\otimes\whii_{\alpha/Y}=\whii_{\beta/Y}$ and so, $\wt{M}_{\beta/Y}\otimes\whii_{\alpha/Y}=\wt{M}_{\beta/Y}$ . 
 From (\ref{reromo-rCech}) we get the exact triangle
 $$
\wt{M}_{\beta/Y}(d_{\alpha})[2d_{\alpha}]\stackrel{v}{\row}\wt{M}_{\beta/Y}[-1]\row
\wt{M}_{\alpha/Y}\otimes\wt{M}_{\beta/Y}\row\wt{M}_{\beta/Y}(d_{\alpha})[2d_{\alpha}+1].
$$
We have: $v^r\in\Hom(\wt{M}_{\beta/Y}(rd_{\alpha})[2rd_{\alpha}+r],\wt{M}_{\beta/Y})=
\Hom(C_{\beta},T(-rd_{\alpha})[-2rd_{\alpha}-r])$. 
The object $(\pi_Y)_{\#}C_{\beta}$ vanishes in the etale
realisation and is compact. Thus, by the Beilinson-Lichtenbaum conjecture it has only finitely many non-zero
diagonals in motivic cohomology. This shows that the morphism $v$ is nilpotent (cf. Proposition \ref{et-comp-fin-diag})
and so, $\wt{M}_{\beta/Y}$ belongs to the thick ideal of $\dmgmEF{Y}{\ff_2}$ generated by $\wt{M}_{\alpha/Y}\otimes\wt{M}_{\beta/Y}$.
Hence, $\wt{M}_{\beta,Y}$ belongs to the thick ideal
of $\dmgmkF{\ff_2}$ generated by
$(\pi_Y)_{\#}(\wt{M}_{\alpha/Y}\otimes\wt{M}_{\beta/Y})$.

Let $C_{\alpha}\boxtimes C_{\beta}\in\dmgmEF{Y\times Y}{\ff_2}$ be the external tensor product. It is a co-algebra in the latter category. By Proposition \ref{reromo-decomp},
$(\pi_{Y\times Y})_{\#}(C_{\alpha}\boxtimes C_{\beta})$ is a direct sum of shifted copies of
$\wt{M}_{\alpha,Y}\otimes\wt{M}_{\beta,Y}$
and $(\pi_Y)_{\#}(C_{\alpha}\otimes C_{\beta})$ is a direct sum of shifted copies of 
$(\pi_Y)_{\#}(\wt{M}_{\alpha/Y}\otimes\wt{M}_{\beta/Y})$. 
Let $f:Y\row Y\times Y$ be the diagonal embedding.
Then $f^*(C_{\alpha}\boxtimes C_{\beta})=C_{\alpha}\otimes C_{\beta}\in\dmgmEF{Y}{\ff_2}$.
By Proposition \ref{coalg-pullback}, $(\pi_Y)_{\#}(C_{\alpha}\otimes C_{\beta})\in
\la (\pi_{Y\times Y})_{\#}(C_{\alpha}\boxtimes C_{\beta})\ra$. Hence, $\wt{M}_{\beta,Y}\in\la\wt{M}_{\alpha,Y}\otimes\wt{M}_{\beta,Y}\ra\subset\la\wt{M}_{\alpha,Y}\ra$.
 \Qed
\end{proof}

Since we already have pull-backs at our disposal, in order to establish independence of our
ideal $\la\wt{M}_{\alpha,Y}\ra$
of the smooth $k$-model $Y$ of $\op{Spec}(E)$, it remains to prove the following statement allowing to pass from an open subvariety to the whole variety.

\begin{prop}
 \label{from-U-to-Y}
 Let $\alpha\in K^M_*(E)/2$ be a pure symbol and $Y$ be a smooth neighbourhood of $\op{Spec}(E)$, where $\alpha$ is unramified. Let $U\subset Y$ be a non-empty open subvariety.
 Then $\wt{M}_{\alpha,Y}\in\la\wt{M}_{\alpha,U}\ra$.
\end{prop}

\begin{proof}
 Induction on the dimension of $Y$. For $\ddim(Y)=0$, there is nothing to prove, as $U=Y$, in this case.
 
 \noindent
 (step) By induction and Gysin triangles, it is sufficient to show that, for any point $x$ of $Y$ with the neighbourhood $X\subset\bar{x}$ of it, $\wt{M}_{\alpha_x,X}\in\la\wt{M}_{\alpha,U}\ra$. 
 It is clearly sufficient to prove it for generic points
 of divisors. So, we may assume that $U$ is a complement in $Y$ to a smooth divisor $Z$. By the induction on the
 dimension and Proposition \ref{pullback-alpha}, it is enough to show that, for some open subvariety $V$ of $Z$, $\wt{M}_{\alpha,V}\in\la\wt{M}_{\alpha,U}\ra$. Thus, we may safely replace $Y$ by any open neighbourhood of $\op{Spec}(k(Z))$ in it, i.e. the
 question is reduced to the respective DVR ${\cal O}_{Y,Z}$. Moreover, we may take any Nisnevich neighbourhood of the mentioned point, as it has the same residue field.
 This reduces it to the Henselization ${\cal O}^h_{Y,Z}$.
 
 \begin{lem}
  \label{Henz-Kk}
  Let $R$ be a DVR of finite type over a field $k$ of $char=0$, with the fraction field $K$ and the residue field $\kappa$. Let $R^h$ be its Henselization. Then there is an embedding
  $\kappa\row R^h$ splitting the projection $R^h\row\kappa$.
  The respective map $\kappa\row K^h$ identifies $K^M_*(\kappa)/2$ with the unramified elements in $K^M_*(K^h)/2$.
 \end{lem}

 \begin{proof}
 The first part is the standard application of the Hensel's Lemma, which also implies that any element of
 $1+{\frak m}\subset K^h$ is a square and so, proves the second statement.
  \Qed
 \end{proof}
 
 Lemma \ref{Henz-Kk} shows that there is a Nisnevich neighbourhood $a:A\row Y$ of $\op{Spec}(k(Z))$ in $Y$, with
 $V=Z\cap A$ and $W=U\times_Y A$, such that there is
 a map $f:A\row V$ splitting the inclusion $V\row A$,
 where $f^*(\alpha_V)=a^*(\alpha_Y)$ (note that since $\alpha$ is unramified on $Y$, it restricts canonically to any point of $Y$).
 By Proposition \ref{pullback-alpha}, 
 $\wt{M}_{f^*(\alpha),W}=\wt{M}_{a^*(\alpha),W}\in\la\wt{M}_{\alpha,U}\ra$. We have a closed-open pair
 $V\hookrightarrow A\hookleftarrow W$ of smooth $V$-varieties (the structure maps given by the restriction of $f$). Let $\eta=\op{Spec}(k(V))$ be the generic point of $V$, and $A_{\eta}$, $W_{\eta}$ be the generic fibers of $f$, respectively, $f_W$. Since $A_{\eta}$ has a rational point (the generic fiber of $V\row A$), its open subvariety $W_{\eta}$ has a zero-cycle of degree $1$. 
 Shrinking $V$, if needed, we get two $V$-subvarieties $V_1\row W$ and $V_2\row W$ of $f_W:W\row V$, where the composition $V_i\stackrel{f_i}{\row}V$ is etale of degrees $m$, respectively $m-1$, for some $m$. Then the composition of the natural maps:
 $$
 \wt{M}_{\alpha/V}\row(f_i)_*f_i^*(\wt{M}_{\alpha/V})=(f_i)_{\#}f_i^*(\wt{M}_{\alpha/V})\row\wt{M}_{\alpha/V}
 $$ 
 is the multiplication by $\ddeg(f_i)$.
 Hence, the difference of classes of the respective maps
 $\wt{M}_{\alpha/V}\row(f_i)_{\#}\wt{M}_{f_i^*(\alpha)/V_i}\row(f_W)_{\#}\wt{M}_{f_W^*(\alpha)/W}$ 
 gives the splitting $\wt{M}_{\alpha/V}\row (f_W)_{\#}(\wt{M}_{f_W^*(\alpha)/W})$ of the natural map $(f_W)_{\#}(\wt{M}_{f_W^*(\alpha)/W})\row\wt{M}_{\alpha/V}$. Since 
 $\wt{M}_{\alpha,V}=(\pi_V)_{\#}(\wt{M}_{\alpha/V})$ and
 $\wt{M}_{f_W^*(\alpha),W}=(\pi_W)_{\#}(\wt{M}_{f_W^*(\alpha)/W})=(\pi_V)_{\#}(f_W)_{\#}(\wt{M}_{f_W^*(\alpha)/W})$, we get that
 $\wt{M}_{\alpha,V}\in\la\wt{M}_{f_W^*(\alpha),W}\ra\subset\la\wt{M}_{\alpha,U}\ra$.
 The induction step and the statement are proven.
 \Qed
\end{proof}

\begin{rem}
This is a more complicated version of the statement claiming that, for a smooth connected $Y$ and an open $U\subset Y$, one has: $\la M(Y)\ra=\la M(U)\ra$. Actually, this ideal depends exactly on the $2$-equivalence class of the field extension $k(Y)/k$.
\end{rem}

Combining Propositions \ref{pullback-alpha} and \ref{from-U-to-Y} we get:

\begin{cor}
 \label{ideal-indep-of-Y}
 The thick tensor ideal $\la\wt{M}_{\alpha,Y}\ra$ doesn't depend on the choice of the smooth neighbourhood $Y$, but only on $\alpha$ itself.
\end{cor}

We also obtain:

\begin{cor}
 \label{orthog-tM-gen-pt}
 Let $N\in\dmgmkD$, $E/k$ be a finitely generated extension and $\alpha\in K^M_*(E)/2$ be a pure symbol. Then the following conditions are equivalent:
 $$
  (1)\hspace{2mm}N\in\wt{M}_{\alpha,Y}^{\perp};\hspace{5mm}
  (2)\hspace{2mm}N_E\in\wt{M}_{\alpha}^{\perp};
  \hspace{5mm}
  (3)\hspace{2mm}N_E\in\la M_{\alpha}\ra.
 $$
\end{cor}

\begin{proof}
 Since a Nisnevich sheaf with transfers is zero at the generic point of some variety if and only if it is zero in some neighbourhood of it, using Proposition \ref{from-U-to-Y} this gives $(2)\Rightarrow(1)$. Suppose now that $N_E\otimes\wt{M}_{\alpha}\neq 0$. Then it has a non-zero motivic cohomology class $w$ over some purely transcendental extension $E(V)$ (as this object is compact). Shrinking $Y_{k(V)}$ to some open $U$, we may assume that all the Brown-Gersten-Quillen differentials on $w$ are trivial, and so, it lifts to a non-trivial element of motivic cohomology of $N_{k(V)}\otimes\wt{M}_{\alpha,U}$. Hence the latter object is non-zero. By Proposition \ref{pullback-alpha}, then so is $N\otimes\wt{M}_{\alpha,Y}$. Thus, $(1)\Rightarrow(2)$.
The equivalence $(2)\Leftrightarrow(3)$ follows from Proposition \ref{orthog-MMt}.
 \Qed
\end{proof}

Finally, we have the control over the thick tensor ideal generated by extended reduced Rost motives under residues.

\begin{prop}
 \label{tM-residue}
 Let $R$ be a DVR of finite type over $k$, with the function field $K$ and the residue field $\kappa$. Let 
 $\beta\in K^M_{n+1}(K)/2$ be a pure symbol and $\alpha=\partial(\beta)\in K^M_n(\kappa)/2$ be its residue. Let $Y$ and $V$ be smooth neighbourhoods of
 $\op{Spec}(K)$ and $\op{Spec}(\kappa)$, where the respective symbols are unramified. Then
 $\wt{M}_{\alpha,V}\in\la\wt{M}_{\beta,Y}\ra$.
\end{prop}

\begin{proof}
 By Corollory \ref{ideal-indep-of-Y}, we may substitute $V$ by any non-empty open subvariety of it. We may assume that $\beta=\{s\}\cdot\alpha'$, where $s$ is a local parameter of our DVR and $\alpha'$ is unramified.
 Using Lemma \ref{Henz-Kk} and Proposition \ref{pullback-alpha}, arguing as in the proof of Proposition \ref{from-U-to-Y}, we may assume that there is a smooth morphism 
 $f:X\row V$ and an open-closed pair $Y\row X\low V$ of $V$-varieties, such that $\alpha'=f^*(\alpha)$. 
 Again, by Corollory \ref{ideal-indep-of-Y}, we may replace $V$ by an arbitrarily small neighbourhood of $\op{Spec}(\kappa)$. Let $\eta$ be the generic point of $V$ and $X_{\eta}$, $Y_{\eta}$ be the generic fibers of
 the respective projections. Here $X_{\eta}$ is a smooth curve over $\op{Spec}(\kappa)$, $v=g(\eta)$ is a rational point on it and $Y_{\eta}$ is the complement to $v$.
 It is sufficient to show that $\wt{M}_{\alpha}\in\la\wt{M}_{\{s\}\cdot f^*(\alpha),Y_{\eta}}\ra\subset\dmgmEF{\kappa}{\ff_2}$. Indeed, then $\wt{M}_{\alpha/V'}\in\la f_{\#}(\wt{M}_{\{s\}\cdot f^*(\alpha)/f^{-1}(V')})\ra\subset\dmgmEF{V'}{\ff_2}$, for some sufficiently small open neighbourhood $V'$ of $\eta$ in $V$. 
By projection formula, it sits in $f_{\#}\la\wt{M}_{\{s\}\cdot f^*(\alpha)/f^{-1}(V')})\ra$.
By Corollary \ref{pi-dash-tensor-ideal},
 $\wt{M}_{\alpha,V'}=(\pi_{V'})_{\#}\wt{M}_{\alpha/V'}\in\la\wt{M}_{\{s\}\cdot f^*(\alpha),f^{-1}(V')}\ra=\la\wt{M}_{\beta,Y}\ra\subset\dmgmkF{\ff_2}$.
 
 \begin{lem}
  \label{odd-deg-Mt}
  Let $j:\op{Spec}(E)\row\op{Spec}(F)$ be an extension of odd degree, $\gamma\in K^M_*(F)/2$ be a pure symbol, and $Y$, $Z$ be smooth neighbourhoods of $\op{Spec}(E)$
  and $\op{Spec}(F)$, where the symbols $j^*(\gamma)$ and $\gamma$ are unramified. Then
  $\la\wt{M}_{j^*(\gamma),Y}\ra=\la\wt{M}_{\gamma,Z}\ra$.
 \end{lem}

 \begin{proof}
  From Corollary \ref{ideal-indep-of-Y} we may assume that
  $j$ extends to an etale morphism $j:Y\row Z$ of odd degree.
  From Proposition \ref{pullback-alpha} we know that
  $\wt{M}_{j^*(\gamma),Y}\in\la\wt{M}_{\gamma,Z}\ra$.
  Finally, the composition
  $$
  \wt{M}_{\gamma/Z}\row j_*j^*(\wt{M}_{\gamma/Z})=j_{\#}j^*(\wt{M}_{\gamma/Z})\row\wt{M}_{\gamma/Z}
  $$
  is the multiplication by $\ddeg(j)$, which is odd. So,
  $\wt{M}_{\gamma,Z}\in\la\wt{M}_{j^*(\gamma),Y}\ra$.
 \Qed
 \end{proof}
 
 The local parameter $s\in K^{\times}$ defines a rational function on the smooth projective model $\ov{X}$ of $X_{\eta}$, which gives a map $j:\ov{X}\row\pp^1_{\kappa}$, such that $j^*(t)=s$, for the standard coordinate $t$
 on $\pp^1$. I claim that $s$ may be modified by a square to make the degree of $j$ odd.
 Indeed, let $D_0$ and $D_{\infty}$ denote the divisor of zeroes, respectively, poles of $s$. Then $D_0=[x]+B$, where $x$ is our point and $B$ doesn't contain $x$. 
 We can find another
 local parameter $s'$, with the divisors of zeroes and poles $D'_0=[x]+B'$, respectively, $D'_{\infty}$, where
 $D'_{\infty}$ doesn't intersect $D_{\infty}$
 and $B$ doesn't intersect $B'$. Then the divisors of
 zeroes, repectively, poles of $s/(s')^2$ will be:
 $B+2D'_{\infty}$, respectively, $[x]+2B'+D_{\infty}$
 and there are no further cancellations. 
 Either $\ddeg(D_0)$ is odd and $s$ gives a map of odd degree, or $\ddeg(B+2D'_{\infty})$ is odd and $s/(s')^2$
 gives a map of odd degree.
 
 Since 
 $j^*(\{t\})=\{s\}\in K^M_1(\kappa(X_{\eta}))/2=K^{\times}/(K^{\times})^2$, by Lemma \ref{odd-deg-Mt}, $\la\wt{M}_{\{s\}\cdot f^*(\alpha),Y_{\eta}}\ra=
 \la\wt{M}_{\{t\}\cdot\alpha,\gm}\ra$ and so, it is enough to prove our result for the case where $Z=\op{Spec}(\kappa)$, $Y=\gm$ and $\beta=\{t\}\cdot\alpha$, for $\alpha\in K^M_n(\kappa)/2$ and $t$ - the coordinate on $\gm$. This follows from the following Proposition.
 \Qed
\end{proof}

\begin{prop}
 \label{Gm-residue}
 Let $\alpha\in K^M_n(k)/2$ be a pure symbol and $t$ be the coordinate on $\gm$. Then
 $\la\wt{M}_{\{t\}\cdot\alpha,\gm}\ra=\la\wt{M}_{\alpha}\ra$.
\end{prop}

\begin{proof}
 It follows from Propositions \ref{pullback-alpha} and
 \ref{divisible-Mt} that $\wt{M}_{\{t\}\cdot\alpha,\gm}\in\la\wt{M}_{\alpha}\ra$. 
 
 To prove the other inclusion, we will treat the cases: $n=0$ and $n>0$ separately. 
 
 \noindent \un{Case $n=0$}: In this case, $\beta=\{t\}$ and the respective Pfister fibration is the ``square map'' $\gm\stackrel{*2}{\row}\gm$. On the level of motives, it is $T\oplus T(1)[1]\stackrel{id\oplus 0}{\lrow}T\oplus T(1)[1]$. Thus, the extended reduced Rost motive $\wt{M}_{\{t\},\gm}$ fits the diagram:
 \begin{equation*}
 \label{nzero-eromo}
 \xymatrix{
 & \gm \ar[d]^{[1]} \ar[rd]^{[1]} & \\
 \gm \ar[ru]^{*2} & R \ar[l] \ar[r] \ar@{}[lu]|-(0.22){\star} \ar@{}[rd]|-(0.22){\star} & \wt{M}_{\{t\},\gm} \ar[ld]^{[1]}\\
 & \gm \ar[u] \ar[lu]^{(*2)^{\vee}} &
 },
 \end{equation*}
 where duality is with respect to $\Homi(-,T(1)[1])$.

 Thus, we have an exact triangle:
 $$
 T(1)\row\wt{M}_{\{t\},\gm}\row T[1]\stackrel{u}{\row}T(1)[1],
 $$
 where $u$ is either $\tau$, or zero. Since $\wt{M}_{\{t\},\gm}$ disappers in \'etale topology, we get that 
 $u=\tau$, and so, $\wt{M}_{\{t\},\gm}=\op{Cone}(T\stackrel{\tau}{\row}T(1))=\wt{M}_{\{\}}(1)$. So, 
 not only ideals, but even reduced Rost motives themselves coincide up to Tate-shift. 
 
 \noindent \un{Case $n>0$}:
 We will identify $\wt{M}_{\{t\}\cdot\alpha,\gm}$ with the cone of a nilpotent map between two shifted copies of $\wt{M}_{\alpha}$. Let's start by computing the motivic cohomology of $\wt{M}_{\{t\}\cdot\alpha,\gm}$. 
 
 The Brown-Gersten-Quillen type spectral sequence gives
 a short exact sequence:
 $$
 0\row\op{Coker}^{j-1,i}\row\moco{j}{i}(\wt{M}_{\{t\}\cdot\alpha,\gm};\ff_2)\row\op{Ker}^{j,i}\row 0,
 $$
 where $\op{Ker}^{j,i}$ and $\op{Coker}^{j,i}$ is the kernel, respectively, cokernel of the map:
 $$
 \moco{j}{i}(\wt{M}_{\beta_{k(t)}};\ff_2)\stackrel{\partial}{\lrow}\operatornamewithlimits{\oplus}_{x\in\gm^{(1)}}
 \moco{j-1}{i-1}(\wt{M}_{\beta_{k(x)}};\ff_2),
 $$
 where $\beta=\{t\}\cdot\alpha$. Here $\wt{M}_{\beta_{k(t)}}$ and $\wt{M}_{\beta_{k(x)}}$ are the usual reduced Rost motives corresponding to the pure symbol $\beta$ restricted to the respective point (note that $\beta$ is unramified on $\gm$, so such specialisations are canonical). 
 
 For a pure symbol $\gamma\in K^M_{n+1}(F)/2$, the motivic cohomology of $\wt{M}_{\gamma}$ is described as follows.
 It is concentrated on $2^n$ diagonals, each isomorphic to $R_{\gamma}=\gamma\cdot K^M_*(F)/2$ up to shift.
 The generators are parametrised by the subsets of $\ov{(n-1)}=[0,1,\ldots,n-1]$.
 More precisely,
 $$
 \moco{*}{*'}(\wt{M}_{\gamma};\ff_2)=\operatornamewithlimits{\oplus}_{I\subset \ov{(n-1)}}r_I\cdot r_n^{-1}\cdot R_{\gamma}.
 $$
 Use \cite[Theorem 3.5, Corollary 3.6]{Iso} and the exact
 triangle (with $d=2^n-1$):
 $$
 \whii_{\gamma}(d)[2d]\stackrel{r_n}{\lrow}\whii_{\gamma}[-1]\lrow\wt{M}_{\gamma}\lrow\whii_{\gamma}(d)[2d+1].
 $$
 It has a natural structure of a module over the motivic homology of $\whii_{\gamma}$. The latter ring is generated over the ring $R_{\gamma}=K^M_*(F)/\kker(\cdot\gamma)$ by elements $r_i$, $0\leq i\leq n$, where
 $\ddeg(r_i)=(1-2^{i})[1-2^{i+1}]$ - see
 \cite[Theorem 3.5]{Iso} (the above $r_I$ is just the product $\prod_{i\in I} r_i$). Moreover, as such a module, it has a single generator: $r_n^{-1}$.
 
 Since $\beta$ is divisible by $\alpha$, for any point $y$
 of $\gm$, $\wt{M}_{\beta_{k(y)}}\otimes\whii_{\alpha_{k(y)}}\cong\wt{M}_{\beta_{k(y)}}$. In particular, the motivic cohomology of $\wt{M}_{\beta_{k(y)}}$ is naturally a module over the motivic homology of $\whii_{\alpha}$ (since the motivic cohomology of $\whii_{\alpha}$ is a module over it - see \cite[Corollary 3.6]{Iso}). The map $\partial$ above is a map of $\moho{*}{*'}(\whii_{\alpha};\ff_2)$-modules.
 This map naturally splits (diagonal-by-diagonal) into a direct sum of maps
 $r_I\cdot r_n^{-1}\cdot \left(R_{\beta_{k(t)}}\stackrel{\partial}{\lrow}\oplus_{x\in\gm^{(1)}}R_{\beta_{k(x)}}\right)$, for $I\subset\ov{(n-1)}$. 
 
 \begin{lem}
  \label{Spr}
  The map $R_{\beta_{k(t)}}\stackrel{\partial}{\lrow}\oplus_{x\in\gm^{(1)}}R_{\beta_{k(x)}}$ is surjective. Its kernel is $\{t\}\cdot\alpha\cdot K^M_*(k)/2$.
 \end{lem}

 \begin{proof}
  For $x\in\gm$, the map $\partial_x:K^M_{*+1}(k(t))/2\row K^M_*(k(x))/2$
  maps $R_{\beta_{k(t)}}=\{t\}\cdot\alpha\cdot K^M_*(k(t))/2$ to $R_{\beta_{k(x)}}$. The map
  $\partial_0:K^M_{*+1}(k(t))/2\row K^M_*(k)/2$ maps
  $R_{\beta_{k(t)}}$ to $R_{\alpha}=\alpha\cdot K^M_*(k)/2$. I claim that the map
  $$
  R_{\beta_{k(t)}}\stackrel{\partial}{\lrow}
  \left(\oplus_{x\in\gm^{(1)}}R_{\beta_{k(x)}}\right)
  \oplus R_{\alpha}
  $$
  is an isomorphism. Indeed, by the Springer's theorem,
  this map is injective. It is sufficient to observe that the image of the restriction $j:K^M_*(k)/2\row K^M_*(k(t))/2$ intersects trivially with $R_{\beta_{k(t)}}$ (since $\partial_0(\{-t\}\cdot j_*(u))=u$, while $\{-t\}\cdot R_{\beta_{k(t)}}=0$).
  To show surjectivity, we need to repeat the arguments of Springer. We start by observing that, for any $u\in R_{\alpha}$, $\partial_0(\{t\}\cdot u)=u$ and $\partial_x(\{t\}\cdot u)=0$, for any $x\in\gm$. Then, by induction on the degree of a point $x$, we show that $R_{\beta_{k(x)}}$ is covered modulo
  points of smaller degree and the origin (i.e., $R_{\alpha}$). Let $p(t)$ be the irreducible polynomials
  of degree $m$ defining the point $x$. Then an element $w$ in 
  $R_{\beta_{k(x)}}=(\{t\}\cdot\alpha)_{k(x)}\cdot K^M_*(k(x))/2$ is a specialisation of an element $v\in R_{\beta_{k(t)}}$ expressed using polynomials in $t$ of degree smaller than $m$. Then 
  $\partial_x(\{p(t)\}\cdot v)=w$, while $\partial_y$
  of this element is zero, for any point $y$ of degree $\geq m$, aside from $x$. Hence, our map $\partial$ is surjective and so, an isomorphism. The Lemma is proven.
  \Qed
 \end{proof}
 
 Thus, we have computed the motivic cohomology of $\wt{M}_{\{t\}\cdot\alpha,\gm}$:
 $$
 \moco{*}{*'}(\wt{M}_{\{t\}\cdot\alpha,\gm},\ff_2)=
 \operatornamewithlimits{\oplus}_{I\subset \ov{(n-1)}}r_I\cdot r_n^{-1}\cdot R_{\alpha}.
 $$
 Note that, as a module over $A=\moho{*}{*'}(\whii_{\alpha};\ff_2)$, it is generated by a single element $r_n^{-1}$. At the same time, the motivic cohomology of $\wt{M}_{\alpha}$ is:
 $$
 \moco{*}{*'}(\wt{M}_{\alpha};\ff_2)=\operatornamewithlimits{\oplus}_{J\subset \ov{(n-2)}}r_J\cdot r_{n-1}^{-1}\cdot R_{\alpha}.
 $$
 So, the former $A$-module is an extension of two copies of the latter one. We will show that the same is true about the motives themselves.
 
 The motive $\wt{M}_{\{t\}\cdot\alpha,\gm}$ is self-dual with respect to $\Homi(-,T(2^n)[2^{n+1}-1])$. Thus, the motivic homology of it has the same structure as motivic cohomology:
 $$
 \moho{*}{*'}(\wt{M}_{\{t\}\cdot\alpha,\gm},\ff_2)=
 \operatornamewithlimits{\oplus}_{I\subset \ov{(n-1)}}r_I\cdot (r_n^{-1})^{\vee}\cdot R_{\alpha}.
 $$
 Since $\wt{M}_{\{t\}\cdot\alpha,\gm}$ is stable under $\otimes\whii_{\alpha}$, we have the identification:
 $$
 \Hom(\whii_{\alpha}(a)[b],\wt{M}_{\{t\}\cdot\alpha,\gm})=\Hom(T(a)[b],\wt{M}_{\{t\}\cdot\alpha,\gm}).
 $$
 In particular, the element $r_{n-1}(r_n^{-1})^{\vee}$
 gives the map $g:\whii_{\alpha}(2^{n-1})[2^n-1]\row
 \wt{M}_{\{t\}\cdot\alpha,\gm}$. 
 We have a distinguished triangle:
 $$
 \whii_{\alpha}(2^{n-1}-1)[2^n-2]\stackrel{r_{n-1}}{\lrow}
 \whii_{\alpha}[-1]\lrow\wt{M}_{\alpha}\lrow\whii_{\alpha}(2^{n-1}-1)[2^n-1].
 $$
 Since $r_{n-1}^2$ has {\it diagonal degree} $2^n$, while the motivic homology of $\wt{M}_{\{t\}\cdot\alpha,\gm}$ is concentrated on the diagonals in the range $[-1,2^n-2]$, $g$ lifts to a map $f:\wt{M}_{\alpha}(2^{n-1})[2^n]\row\wt{M}_{\{t\}\cdot\alpha,\gm}$. By the same degree considerations, the lifting is unique.
 
 The map $\whii_{\alpha}[-1]\row\wt{M}_{\alpha}$ is surjective on motivic homology and maps the unit $T[-1]\row\whii_{\alpha}[-1]$ to $(r_{n-1}^{-1})^{\vee}$, so $f_*((r_{n-1}^{-1})^{\vee})=r_{n-1}(r_n^{-1})^{\vee}$, by construction. Hence, $f_*$ identifies
 $$
 \moho{*}{*'}(\wt{M}_{\alpha};\ff_2)=\!\!\!\operatornamewithlimits{\oplus}_{J\subset \ov{(n-2)}}\!\!r_J\cdot (r_{n-1}^{-1})^{\vee}\cdot R_{\alpha}
 \hspace{2mm}\text{with}\hspace{2mm}
 \operatornamewithlimits{\oplus}_{J\subset \ov{(n-2)}}r_J\cdot r_{n-1} (r_{n}^{-1})^{\vee}\cdot R_{\alpha}
 \subset\moho{*}{*'}(\wt{M}_{\{t\}\cdot\alpha,\gm};\ff_2).
 $$
 The map $\whii_{\alpha}[-1]\row\wt{M}_{\alpha}$ is injective on motivic cohomology and $g^*(r_n^{-1})=r_{n-1}^{-1}$. Hence, $f^*$ identifies
 $$
 \operatornamewithlimits{\oplus}_{J\subset \ov{(n-2)}}r_J\cdot r_{n}^{-1}\cdot R_{\alpha}
 \subset\moco{*}{*'}(\wt{M}_{\{t\}\cdot\alpha,\gm};\ff_2)
 \hspace{2mm}\text{with}\hspace{2mm}
 \moco{*}{*'}(\wt{M}_{\alpha};\ff_2)=\!\!\!\operatornamewithlimits{\oplus}_{J\subset \ov{(n-2)}}\!\!r_J\cdot r_{n-1}^{-1}\cdot R_{\alpha}.
 $$
 So, $f_*$ (respectively, $f^*$) identifies motivic homology (respectively, cohomology) of $\wt{M}_{\alpha}$
 with the half of motivic homology/cohomology of $\wt{M}_{\{t\}\cdot\alpha,\gm}$. 
 Let $f^{\vee}:\wt{M}_{\{t\}\cdot\alpha,\gm}\row
 \wt{M}_{\alpha}(1)[1]$ be the dual map. Then $(f^{\vee})_*$ and $(f^{\vee})^*$ identify the other half of homology/cohomology of $\wt{M}_{\{t\}\cdot\alpha,\gm}$ with that of $\wt{M}_{\alpha}$. Note that this property holds not only over the ground field, but also over any extension of it. Hence, $f$ identifies $\wt{M}_{\alpha}(2^{n-1})[2^n]$ with the piece $\tau_{>2^{n-1}}(\wt{M}_{\{t\}\cdot\alpha,\gm})$ of the homotopy $t$-structure filtration,
 while $f^{\vee}$ identifies $\tau_{\leq 2^{n-1}}(\wt{M}_{\{t\}\cdot\alpha,\gm})$ with $\wt{M}_{\alpha}(1)[1]$
 (recall that motivic homology (considered as a Rost cycle module \cite{RCM}, i.e., over all field extensions) of $\tau_{>i}(N)$ is identified with the diagonal $>i$ part of motivic homology of $N$, and similar for $\tau_{\leq i}(N)$). Hence, we have an exact triangle:
 \begin{equation*}
  \label{tMbeta-tMalpha-tri}
  \wt{M}_{\alpha}(2^{n-1})[2^n]\stackrel{f}{\lrow}\wt{M}_{\{t\}\cdot\alpha,\gm}\stackrel{f^{\vee}}{\lrow}\wt{M}_{\alpha}(1)[1]\stackrel{\ffi}{\lrow}\wt{M}_{\alpha}(2^{n-1})[2^n+1].
 \end{equation*}

 It remains to observe that $\wt{M}_{\alpha}$ disappears in the etale topology and
 apply the following useful fact (Proposition \ref{et-comp-fin-diag}) to conclude
 that $\ffi$ is nilpotent. Hence, $\wt{M}_{\alpha}\in\la\wt{M}_{\{t\}\cdot\alpha,\gm}\ra$.
 
 Proposition \ref{Gm-residue} is proven.
 \Qed
\end{proof}

\begin{prop}
 \label{et-comp-fin-diag}
 Let $M\in\dmgmkF{\Lambda}$. Then
 \begin{itemize}
  \item[$(1)$] $M_{\et}=0$ $\Leftrightarrow$ $M$ has only finitely many non-zero homology objects in the homotopy $t$-structure;
  \item[$(2)$] If $M_{\et}=0$ and $M\stackrel{\ffi}{\lrow}M(a)[b]$ is some map with $a\neq b$, then $\ffi$ is nilpotent.
 \end{itemize}
\end{prop}

\begin{proof}
 (1) 
 Homology objects of $M$ with respect to the homotopy $t$-structure are Rost cycle modules given by diagonals in motivic homology of $M$ (\cite{Deg}). Since $M$ is compact, we may 
 substitute it by the motivic cohomology of $M^{\vee}$ instead. Since $M^{\vee}$ is compact, it has no motivic cohomology with numbers $>d$, for some $d$. Hence, we will have only finitely many such non-zero diagonals if and only if diagonals with numbers $<<0$ are trivial. By the Beilinson-Lichtenbaum conjecture (\cite[Theorem 6.18]{Vo-Z-l}), the latter is equivalent to the fact that $M_{\et}=0$.
 
 \noindent
 (2) If $M_{\et}=0$, then so is $(M\otimes M^{\vee})_{\et}$. Hence, by (1), this object has only finitely many diagonals in motivic homology. Since $\ffi$ ``moves'' in non-diagonal direction, some power of it will be zero
 (as it is represented by homology class $T(-ra)[-rb]\row M\otimes M^{\vee}$). 
 \Qed
\end{proof}

\section{Invariants of prime ideals}
 \label{three}

Now we can introduce a certain coordinate system on the Balmer spectrum of geometric motives.
With the help of {\it extended reduced Rost motives} we will define some invariants of prime ideals of Voevodsky category which will allow us to study specialisation relation among them. Our invariants will take values in the subsets of $\op{Pure}$.

For an object $A$ of $\op{DM}_{gm}(k;\ff_2)$, we will denote as
$A^{\perp}$ the collection of objects $B$ of this category, such that $A\otimes B=0$.
For $(E,\alpha)\in\op{Pure}$, we will denote as $\wt{M}_{\alpha,Y}$ the respective extended reduced Rost motive.

\begin{defi}
Let ${\frak a}\subset\op{DM}_{gm}(k;\ff_2)$ be a prime ideal. Define:
\begin{equation*}
 \begin{split}
  &{\cal G}({\frak a})=\{(E,\alpha)\in\op{Pure}\,|\,(\wt{M}_{\alpha,Y})^{\perp}\subset{\frak a}\};\\
  &{\cal H}({\frak a})=\{(E,\alpha)\in\op{Pure}\,|\,
  \wt{M}_{\alpha,Y}\in{\frak a}\}.
 \end{split}
\end{equation*}
\end{defi}

Note that since ${\frak a}$ is prime, ${\cal G}({\frak a})\cup{\cal H}({\frak a})=\op{Pure}$.

Let us compute these invariants for isotropic points. Such points are parametrised by the
$2$-equivalence classes of field extension, where the point corresponding to the extension $F/k$ is denoted
${\frak a}_{F}$ (see \cite[Theorem 5.13]{INCHKm}).
Such an ideal is the pre-image under the natural restriction map $\op{DM}_{gm}(k;\ff_2)\row\op{DM}_{gm}(\wt{F};\ff_2)$ of the thick tensor ideal generated by motives of all $2$-anisotropic varieties. Here $\wt{F}=F(\pp^{\infty})$ is the {\it flexible closure} of the
field $F$.

\begin{prop}
 \label{Ha-iso-descr}
 Let $F/k$ be some field extension, $(E,\alpha)\in\op{Pure}$ and $\wt{M}_{\alpha,Y}$ be the respective
 extended reduced Rost motive. Let $P$ be a smooth projective model for $E/k$. Then
 $$
 \wt{M}_{\alpha,Y}\in{\frak a}_{F}\hspace{2mm}
 \Leftrightarrow\hspace{2mm}
 \text{either \ }a)\,\,\alpha_{F(P)}=0,\text{ or \ }b)\,\, P_F\text{ is anisotropic}.
 $$
\end{prop}

\begin{proof}
 $(\low)$ If $\alpha_{F(P)}=0$, then by Proposition 
 \ref{triv-erRm}, $(\wt{M}_{\alpha,Y})_{\wt{F}}=0$.
 Hence, $\wt{M}_{\alpha,Y}\in{\frak a}_{F}$.\\
 If $P_F$ is anisotropic, then so is $P_{\wt{F}}$. 
 But $(\whii_{P})_{\wt{F}}$ vanishes at every point of
 $Y_{\wt{F}}$. So, $(\wt{M}_{\alpha,Y})_{\wt{F}}\otimes
 (\whii_{P})_{\wt{F}}=0$. Thus, $\wt{M}_{\alpha,Y}\in{\frak a}_{F}$ (see \cite[Remark 2.8]{Iso}).\\
 $(\row)$ If $P_F$ is isotropic, then, for any anisotropic variety $R/\wt{F}$, its restriction $R_{\wt{F}(P)}$ is still anisotropic. If also $\alpha_{F(P)}\neq 0$, then
 $\whii_{\alpha_{\wt{F}(P)}}\otimes\whii_{R_{\wt{F}(P)}}\neq 0$. Note that $(\whii_{\alpha,Y})_{\wt{F}}\otimes\whii_R$ is just the reduced motive of the \v{C}ech
 simplicial scheme corresponding to the smooth morphism
 $Q_{\alpha_{\wt{F}}}\coprod(R\times Y_{\wt{F}})\row Y_{\wt{F}}$, whose generic fiber is exactly
 $\whii_{\alpha_{\wt{F}(P)}}\otimes\whii_{R_{\wt{F}(P)}}$. But the triviality of the motive of a \v{C}ech simplicial scheme (over a base) is equivalent to the triviality of the generic fiber of it:
 
 \begin{lem}
  \label{rCss-gen-fib}
  Let $Q\row Y$ be a smooth morphism, with $Y$ smooth connected, with the generic fiber $Q_{\eta}\row\eta$. Then \ 
  $
  \whii_{Q/Y}=0\hspace{2mm}\Leftrightarrow\hspace{2mm}
  \whii_{Q_{\eta}/\eta}=0.
  $
 \end{lem}

 \begin{proof}
  $(\low)$ If $\whii_{Q_{\eta}/\eta}=0$, then the generic fiber is isotropic, so all fibers are isotropic. Hence,
  the projection $M(Q)\row M(Y)$ has a splitting and so, $\whii_{Q/Y}=0$.\\
  $(\row)$ If $\whii_{Q/Y}=0$, then the projection
  $\hii_{Q/Y}\row M(Y)$ has a splitting $s$ (an inverse).
  Then the composition $M(\eta)\row M(Y)\stackrel{s}{\row}\hii_{Q/Y}$ factors through the fiber $\hii_{Q_{\eta}/\eta}$ over the generic point, since
  $\op{Hom}(M(\eta),\hii_{Q_D/D}(1)[2])=0$, for smooth divisors $D$ (as motivic homology of smooth simplicial schemes are zero below the zeroth diagonal). Thus, we get the splitting of the generic fiber and so, $\whii_{Q_{\eta}/\eta}=0$.
  \Qed
 \end{proof}
 
 Lemma \ref{rCss-gen-fib} shows that $(\whii_{\alpha,Y})_{\wt{F}}\otimes\whii_{R}\neq 0$,
 for any anisotropic $R$ over $\wt{F}$. Hence, the isotropic
 projector $\Upsilon_{\wt{F}/\wt{F}}$ (\cite[Definition 2.4]{Iso}) doesn't annihilate our motive of the reduced \v{C}ech simplicial scheme:
 $(\whii_{\alpha,Y})_{\wt{F}}\otimes\Upsilon_{\wt{F}/\wt{F}}\neq 0$. In particular, the motivic homology of the latter object is non-zero (over some extension of $\wt{F}$). But such motivic homology can't be $(d)[2d+1]$
 periodic (if non-zero), since it vanishes below the zeroth diagonal. Hence, $(\wt{M}_{\alpha,Y})_{\wt{F}}\otimes\Upsilon_{\wt{F}/\wt{F}}$ also has non-zero homology and so, is non-zero. Thus, $\wt{M}_{\alpha,Y}\not\in{\frak a}_{F}$.
 \Qed
\end{proof}

\begin{prop}
 \label{GH-iso-empty}
 Let $F/k$ be any extension. Then
 \begin{equation*}
  \begin{split}
   &{\cal G}({\frak a}_{F})\cap{\cal H}({\frak a}_{F})=\emptyset\hspace{5mm}\text{and}\\
   &{\cal H}({\frak a}_{F})=\{(E=k(P),\alpha)\in\op{Pure}\,|\,\text{ either }\alpha_{F(P)}=0,\text{ or }P_F\text{ is anisotropic}\}.
  \end{split}
 \end{equation*}
\end{prop}

\begin{proof}
 The description of ${\cal H}({\frak a}_{F})$ follows from Proposition \ref{Ha-iso-descr}.
 
 If $\alpha_{F(P)}=0$, then there exists a smooth $Q$, such that $\alpha_{k(P\times Q)}=0$ and $k(Q)$ is a subfield of $F$. In particular, $Q_F$ is isotropic.
 Since $\wt{M}_{\alpha_{k(P\times Q)}}=0$, by Proposition
 \ref{triv-erRm}, $\wt{M}_{\alpha,Y}|_{k(Q)}=0$, which is equivalent to: $M(Q)\otimes\wt{M}_{\alpha,Y}=0$. At the   
 same time, $M(Q)\not\in{\frak a}_{F}$, since a Tate-motive splits off from it over $F$. Hence, $(E,\alpha)\not\in{\cal G}({\frak a}_{F})$. 
 
 Suppose, $P_F$ is anisotropic. By \cite[Theorem 2.3]{VPS}, there exists a non-zero pure symbol $\beta\in k^M_*(\wt{k})$, such that, for any extension $L/k$, 
 $P_L$ is isotropic $\Leftrightarrow$ $\beta_{\wt{L}}=0$.
 In particular, $\beta_{\wt{k}(P)}=0$, but $\beta_{\wt{F}}\neq 0$. Then $\wt{M}_{\beta,Y'}|_{\wt{k}(P)}=0$, by Proposition \ref{triv-erRm}. Since $k(Y)=k(P)$
 and $k(Q_{\alpha})$ is an extension of it, while $\wt{M}_{\alpha,Y}$ is an extension of a direct summand of $M(Q_{\alpha})$ and two (shifted) copies of $M(Y)$, we get: 
 $\wt{M}_{\beta,Y'}\otimes\wt{M}_{\alpha,Y}=0$. At the same time, $\wt{M}_{\beta,Y'}\not\in{\frak a}_{F}$,
 by Proposition \ref{Ha-iso-descr} (note that the respective smooth model $P'$ is rational and so, isotropic over $k$). Hence, $(E,\alpha)\not\in{\cal G}({\frak a}_{F})$.
  Thus, ${\cal G}({\frak a}_{F})\cap{\cal H}({\frak a}_{F})=\emptyset$.
 \Qed
\end{proof}

We may describe the isotropic ideals completely in terms of their ${\cal G}$-invariants.

\begin{prop}
 \label{iso-ideal-descr}
 Let $F/k$ be a field extension. Then
 $$
 {\frak a}_{F}=\bigcup_{(E,\alpha)\in{\cal G}({\frak a}_{F})}(\wt{M}_{\alpha,Y})^{\perp}.
 $$
\end{prop}

\begin{proof}
 The inclusion $\supset$ follows by the definition of ${\cal G}$. Conversely,
 let $U\in{\frak a}_{F}$, then $U_{\wt{F}}$ is expressible in terms of finitely many motives of anisotropic varieties using cones and direct summands.
 All the varieties and maps involved are defined over some finitely generated extension. So, there exists
 a smooth projective variety $P/k$, such that $k(P)\subset F$, and smooth projective varieties
 $Q_1,\ldots,Q_r$ over $\wt{k}(P)$, such that $Q_i|_{\wt{F}}$ are anisotropic and $U_{\wt{k}(P)}$ is expressed in terms of $Q_i$s. Let 
 $\displaystyle Q=\coprod_{i=1}^r Q_i$. By \cite[Theorem 2.3]{VPS}, there exists a pure symbol $\beta\in k^M_*(\wt{\wt{k}}(P))$ describing the isotropy of $Q$.
 In particular, $\beta_{\wt{\wt{F}}}\neq 0$ and 
 $Q_{\wt{\wt{k}}(P)}$ is a subvariety of $Q_{\beta}$. The  
 latter fact implies that $U_{\wt{\wt{k}}(P)}$ is expressible in terms of the Rost motive $M_{\beta}$. 
 In reality, $\beta$ is defined and the above holds already over some finitely generated extension $k'(P)$
 of $k$,
 where $k'=k(A)$ is purely transcendental. 
 Since $M_{\beta}\otimes\wt{M}_{\beta}=0$, the tensor product $U\otimes$ annihilates the generic fiber of 
 $\wt{M}_{\beta,Y}$. Hence, by Corollary \ref{orthog-tM-gen-pt}, $U\in(\wt{M}_{\beta,Y})^{\perp}$. We have: $P_{F}$ is isotropic
 and so, $\beta_{F(A)(P)}\neq 0$. Thus,
 $(k'(P),\beta)\in{\cal G}({\frak a}_{F})$, by Proposition \ref{GH-iso-empty}. The inclusion $\subset$
 follows.
 \Qed
\end{proof}

This permits to describe the points of the Balmer spectrum specialising into a given isotropic one in terms of the ${\cal G}$-invariant.

\begin{cor}
 \label{incl-iso-pts}
 For any prime ideal ${\frak a}$ and an isotropic ideal 
 ${\frak a}_{F}$, we have:
 $$
 {\frak a}_{F}\subset{\frak a}
 \hspace{2mm}\Leftrightarrow\hspace{2mm}
 {\cal G}({\frak a}_{F})\subset{\cal G}({\frak a}).
 $$
\end{cor}

\begin{proof}
 Follows straight from the definition of ${\cal G}$ and Proposition \ref{iso-ideal-descr}.
 \Qed
\end{proof}

As an immediate corollary, we get:

\begin{thm}
 \label{iso-closed}
 All isotropic points ${\frak a}_{F}$ of the Balmer spectrum $\op{Spc}(\op{DM}_{gm}(k,\ff_2))$ are closed.
\end{thm}

\begin{proof}
 Recall, that a point ${\frak a}$ of the Balmer spectrum is a specialisation of a point ${\frak b}$
 if and only if ${\frak a}\subset{\frak b}$ - see \cite[Proposition 2.9]{Bal-1}.
 
 Suppose, ${\frak a}$ is a specialisation of some isotropic point ${\frak a}_{F}$, i.e. ${\frak a}\subset{\frak a}_{F}$. Then ${\cal G}({\frak a})\subset{\cal G}({\frak a}_{F})$ and ${\cal H}({\frak a})\subset{\cal H}({\frak a}_{F})$. But 
 ${\cal G}({\frak a}_{F})\cap{\cal H}({\frak a}_{F})=\emptyset$, by Proposition \ref{GH-iso-empty}, and ${\cal G}({\frak a})\cup{\cal H}({\frak a})=\op{Pure}$. Thus, ${\cal G}({\frak a})={\cal G}({\frak a}_{F})$ and ${\cal H}({\frak a})={\cal H}({\frak a}_{F})$. 
 By Corollary \ref{incl-iso-pts}, ${\frak a}_{F}\subset{\frak a}$. Hence, ${\frak a}={\frak a}_{F}$.
 Thus, the point ${\frak a}_{F}$ is closed.
 \Qed
\end{proof}

\begin{rem}
 In particular, there are no specialisation relations among distinct isotropic points.
\end{rem}

Denote as $\comg(\frak{a})$ the complement 
$\op{Pure}\!\backslash\calg(\frak{a})$. We have the embedding
$\comg(\frak{a})\subset\calh(\frak{a})$.

\begin{thm}
 \label{Gbar-H-LRCM}
 The subsets $\comg(\frak{a})$ and $\calh(\frak{a})$ are ``light Rost cycle submodules`` of $\op{Pure}$. That is, these are stable under: restriction of fields, residues w.r.to DVRs and action of ${\cal O}^*$.
\end{thm}

\begin{proof}
 Let $\alpha\in K^M_*(E)/2$ and $\beta\in K^M_*(F)/2$ be
 pure symbols (defined over some finitely generated extensions of the base field), where $\beta$ is obtained from $\alpha$ using operations: 1) restriction of fields,
 2) residues with respect to DVRs, and 3) action of ${\cal O}^*$.
 Let $Y$ (respectively $Z$) be smooth open neighbourhoods of $\op{Spec}(E)$ (respectively, $\op{Spec}(F)$), where
 the respective symbols are unramified. It follows from Corollary \ref{ideal-indep-of-Y} and Propositions \ref{pullback-alpha}, \ref{divisible-Mt},
 \ref{tM-residue} that $\wt{M}_{\beta,Z}\in\la\wt{M}_{\alpha,Y}\ra$. Hence, $\alpha\in\calh(\frak{a})\,\Rightarrow\,\beta\in\calh(\frak{a})$. Since the operation $(-)^{\perp}$ reverses unclusions, we also get
 that $\alpha\in\comg(\frak{a})\,\Rightarrow\beta\in\comg(\frak{a})$.
 \Qed
\end{proof}

To any point ${\frak a}$ of the Balmer spectrum
we may assign a $\stackrel{2}{\sim}$-equivalence class
$K({\frak a})$ of field extensions and so, an isotropic point:
$$
K({\frak a})\hspace{3mm}
\stackrel{2}{\sim}\hspace{3mm}\operatornamewithlimits{\ast}_{(E,\alpha)\not\in {\cal H}({\frak a})}E
\hspace{5mm}\stackrel{2}{\sim}\hspace{5mm}
\operatornamewithlimits{\ast}_{(E,1)\not\in {\cal H}({\frak a})}E.
$$
In the case of an isotropic point, it recovers the original point.

\begin{prop}
 For any field extension $F/k$, 
 $K({\frak a}_F)\stackrel{2}{\sim}F$.
\end{prop}

\begin{proof}
By Proposition
\ref{GH-iso-empty}, 
$$
{\cal H}({\frak a}_{F})=\{(E=k(P),\alpha)\in\op{Pure}\,|\,\text{ either }\alpha_{F(P)}=0,\text{ or }P_F\text{ is anisotropic}\}.
$$
In particular, $(E=k(P),1)\not\in{\cal H}({\frak a}_F)$ if and only if $P_F$ is isotropic.
Thus, $K({\frak a}_F)$ is the composite of $k(P)$
for varieties $P$ which become isotropic over $F$. As $F$ is a colimit 
$\op{colim} k(Q)$ of finitely generated extensions, we get:
$K({\frak a}_F)\stackrel{2}{\sim}F$
 \Qed
\end{proof}

It appears that, for any point ${\frak a}$,
${\cal H}({\frak a})$ always contains ${\cal H}$ of some
isotropic points.

\begin{prop}
 \label{H-any-iso}
 Let $F=\op{colim}F_{\lambda}$, where $F_{\lambda}=K({\frak a})(Q_{\lambda})$, where $(k(Q_{\lambda}),1)\in{\cal G}({\frak a})$ (in other words, $M(Q_{\lambda})^{\perp}\subset{\frak a}$).  Then 
 ${\cal H}({\frak a}_F)\subset{\cal H}({\frak a})$.
\end{prop}

\begin{proof}
 Recall that $(k(P),\alpha)\in{\cal H}({\frak a}_F)
 \Leftrightarrow\,\text{either}\,\alpha_{F(P)}=0,\,\,\text{or}\,\,P_F\,\text{is anisotropic}$. 
 
 If $P_F$ is anisotropic, then $P_{K({\frak a})}$ is
 anisotropic $\Rightarrow$ $(k(P),1)\in{\cal H}({\frak a})$ $\Rightarrow$ $(k(P),\alpha)\in{\cal H}({\frak a})$.
 
 If $\alpha_{F(P)}=0$, then $\alpha_{K({\frak a})(P\times Q)}=0$, for some $Q$ with $M(Q)^{\perp}\subset{\frak a}$.
 So, there exists $R$, such that $\alpha_{k(P\times Q\times R)}=0$
 and $(k(R),1)\not\in{\cal H}({\frak a})$, that is,
 $\wt{M}_{\{\}}\otimes M(R)\not\in{\frak a}$, in particular, $M(R)\not\in{\frak a}$.
 The former implies that $\wt{M}_{\alpha,Y}\otimes
 M(Q)\otimes M(R)=0$ and so,
 $\wt{M}_{\alpha,Y}\otimes M(R)\in  M(Q)^{\perp}
 \subset{\frak a}$. Since $M(R)\not\in{\frak a}$ and 
 ${\frak a}$ is prime, we have: 
 $\wt{M}_{\alpha,Y}\in{\frak a}$.
 Hence, $(k(P),\alpha)\in{\cal H}({\frak a})$.
 \Qed
\end{proof}

\begin{rem}
 \begin{itemize}
  \item[$(1)$]
The conditions of the Proposition are satisfied for
$Q=\op{Spec}(k)$, so ${\cal H}({\frak a}_{K({\frak a})})\subset
{\cal H}({\frak a})$.
  \item[$(2)$]
 Hypothetically, Proposition \ref{H-any-iso} describes
 all isotropic points, whose ${\cal H}$ is contained in ${\cal H}({\frak a})$.
Note that, in contrast to Corollary \ref{incl-iso-pts}, this condition doesn't imply a specialisation relation.
 \end{itemize}
\end{rem}

\begin{defi}
We say that a point ${\frak a}$ of the Balmer spectrum is of a "boundary type", if 
$$\ov{{\cal G}}({\frak a})={\cal H}({\frak a}).$$
\end{defi}

\begin{exa}
 \label{exa-bt-iso-et}
 \begin{itemize}
  \item[$(1)$] Any isotropic point ${\frak a}_F$ is of a {\it boundary type} by Proposition
\ref{GH-iso-empty}.
  \item[$(2)$] Let ${\frak a}_{et}$ be the \'etale point. Since ${\frak a}_{et}=\la\op{Cone}(\tau)\ra=\la\wt{M}_{\{\}}\ra$, we see that ${\cal H}({\frak a}_{et})=\op{Pure}$. Denote :
$$
\kker(\ov{k}/k):=\{(E=k(P),\alpha)\in\op{Pure}\,|\,\alpha_{\ov{k}(P)}=0\}.
$$
If $(E=k(P),\alpha)\in\kker(\ov{k}/k)$, then there exists $Q$, such that $\alpha_{k(P\times Q)}=0$, 
and so, $M(Q)\in (\wt{M}_{\alpha,Y})^{\perp}$. Since a Tate-motive splits off from $M(Q)_{\ov{k}}$, we have: $M(Q)\not\in {\frak a}_{et}$ $\Rightarrow$ 
$(\wt{M}_{\alpha,Y})^{\perp}\not\subset {\frak a}_{et}$ $\Rightarrow$
$(E,\alpha)\in\ov{{\cal G}}({\frak a}_{et})$. Thus,
$\kker(\ov{k}/k)\subset\ov{{\cal G}}({\frak a}_{et})\subset\op{Pure}$.
Moreover, both inclusions are proper. For the right one, it is enough to observe that 
$\ov{{\cal G}}({\frak a}_{et})$ doesn't contain "units". Indeed, for $(k(P),1)$, we may choose
$\wt{M}_{\alpha,Y}=M(P)\otimes\wt{M}_{\{\}}$. Here the functor $\otimes\wt{M}_{\{\}}$
is conservative and so, $(\wt{M}_{\alpha,Y})^{\perp}=M(P)^{\perp}$. The latter ideal is contained
in ${\frak a}_{et}$, since a Tate-motive splits off from $M(P)$ in etale realisation. Hence, 
$(k(P),1)\not\in\ov{{\cal G}}({\frak a}_{et})$.
For the left one, consider a numerically trivial Chow motive $N$. By \cite[Corollary 4.6]{VPS},
the numerical triviality of $N$ is controlled by some pure symbol over the flexible closure, that is,
there exists a purely transcendental extension $k(A)/k$ and a pure symbol $\alpha\in K^M_*(k(A))/2$, such that, for any extension $L/k$, the motive $N_L$ is numerically trivial if and only if $\alpha_{L(A)}\neq 0$. In addition, $N_{k(A)}\otimes\wt{M}_{\alpha}=0$. Let now $N$ be such that $N_{\ov{k}}$ is still numerically trivial and
$N$ doesn't vanish in the etale realisation. Then, for the respective symbol $\alpha$, we have:
$\alpha\not\in\kker(\ov{k}/k)$ and $N_{k(A)}\otimes\wt{M}_{\alpha}=0$. 
By Corollary \ref{orthog-tM-gen-pt}, $N\in\wt{M}_{\alpha,Y}^{\perp}$. Since $N$ doesn't vanish
in the etale realisation, we get: $\alpha\in\ov{{\cal G}}({\frak a}_{et})$ and so,
the left inclusion is a proper one. As an example of such a motive $N$, we may choose the
middle part of the motive of an elliptic curve without complex multiplication - see 
\cite[Example 2.13]{Iso}. Such a motive exists over any field $k$ of characteristic zero.
Another choice is a {\it torsion motive} in the sense of \cite{TM}, i.e. a Chow motive whose identity map is annihilated by a natural number. Such motives $\hat{N}$ exist, in particular, as direct summands of Burniat surfaces and $id_{\hat{N}}$ is killed by $2$, in this case - see \cite{GO}. The Burniat surface and the projector are always defined over some finite extension $L/k$ and we
may consider the $2$-torsion motive $N=\pi_{\#}(\hat{N})$, where $\pi:\op{Spec}(L)\row\op{Spec}(k)$. The singular cohomology of $N$ with $\ff_2$-coefficients is non-trivial and so, $N\not\in{\frak a}_{et}$. On the other hand, $N$ is still torsion and so, numerically trivial over $\ov{k}$.
Thus,
$$
\kker(\ov{k}/k)\subsetneq\ov{{\cal G}}({\frak a}_{et})
\subsetneq\op{Pure}.
$$
In particular, ${\frak a}_{et}$ is not of a boundary type.
 \end{itemize}
\end{exa}

In conclusion, let me formulate a couple of natural questions:

\begin{que}
 \begin{itemize}
  \item[$(1)$] Do ${\cal G}$-${\cal H}$-invariants distinguish the points of the Balmer spectrum?
  \item[$(2)$] Are the closed points of the spectrum exactly the points of the boundary type?
 \end{itemize}
\end{que}

\end{document}